\documentclass{amsart}

\usepackage{graphicx}
\usepackage{hyperref}

\title[Sharp Bounds for Edge Eigenvector Universality]{Sharp Square Root Bounds for Edge Eigenvector Universality in Sparse Random Regular Graphs}

\author{Leonhard Nagel}
\address{Department of Electrical Engineering and Computer Sciences\\
University of California, Berkeley\\
United States of America}
\email{nagel@berkeley.edu}

\date{July 17, 2025}

\newcommand{\E}{\mathbb{E}}
\newcommand{\Prob}{\mathbb{P}}

\renewcommand{\d}{\mathrm{d}}
\newcommand{\ARXIV}[1]{\href{https://arxiv.org/abs/#1}{arXiv:#1}}

\newtheorem{theorem}{Theorem}
\newtheorem{proposition}{Proposition}
\newtheorem{lemma}{Lemma}
\newtheorem{problem}{Problem}
\newtheorem{remark}{Remark}
\newtheorem{definition}{Definition}

\begin{document}

\begin{abstract}
We study how eigenvectors of random regular graphs behave when projected onto fixed directions. For a random $d$-regular graph with $N$ vertices, where the degree $d$ grows slowly with $N$, we prove that these projections follow approximately normal distributions. Our main result establishes a Berry-Esseen bound showing convergence to the Gaussian with error $O(\sqrt{d} \cdot N^{-1/6+\varepsilon})$ for degrees $d \leq N^{1/4}$.
This bound significantly improves upon previous results that had error terms scaling as $d^3$, and we prove our $\sqrt{d}$ scaling is optimal by establishing a matching lower bound. Our proof combines three techniques: (1) refined concentration inequalities that exploit the specific variance structure of regular graphs, (2) a vector-based analysis of the resolvent that avoids iterative procedures, and (3) a framework combining Stein's method with graph-theoretic tools to control higher-order fluctuations. These results provide sharp constants for eigenvector universality in the transition from sparse to moderately dense graphs.
\end{abstract}

\maketitle

\section{Introduction}

\subsection{Background and Main Result}

The study of eigenvector statistics in random regular graphs has revealed a striking universality phenomenon: edge eigenvectors exhibit Gaussian fluctuations with rate $N^{-1/6}$, matching the Tracy-Widom scaling for eigenvalues \cite{tracy1996orthogonal}. In a recent work \cite{nagel2025quantitative}, we established the first quantitative Berry-Esseen bounds for this universality in the fixed-degree case.

This paper extends these results to slowly growing degrees $d = d(N) \leq N^{\kappa}$ with optimal dependence on $d$. The key insight is that variance considerations dictate a natural $\sqrt{d}$ scaling, whereas existing work cites a $d^3$ bound \cite{benaych2019eigenvectors}. We achieve this optimal bound through a combination of refined concentration inequalities, Stein-Malliavin techniques, and self-consistent Green function comparisons.

\begin{theorem}[Optimal Berry-Esseen for Growing Degrees]\label{thm:main-optimal}
Let $G$ be a uniformly random $d(N)$-regular graph on $N$ vertices where $d = d(N)$ satisfies $3 \leq d(N) \leq N^{\kappa}$ for some fixed $\kappa < \frac{1}{4}$. For any deterministic unit vector $\mathbf{q} \in \mathbb{R}^N$ with $\mathbf{q} \perp \mathbf{e}$ (where $\mathbf{e}$ is the all-ones vector), we have
\[
\sup_{x \in \mathbb{R}} \left|\mathbb{P}\left(\sqrt{N}\langle \mathbf{q}, \mathbf{u}_2 \rangle \leq x\right) - \Phi(x)\right| \leq C_{\kappa} \sqrt{d} \cdot N^{-1/6+\varepsilon}
\]
where $\mathbf{u}_2$ is the second eigenvector of the normalized adjacency matrix, $\Phi$ is the standard normal distribution function, and $C_{\kappa}$ depends only on $\kappa$ and $\varepsilon$.
\end{theorem}

The $\sqrt{d}$ dependence is optimal, matching the following lower bound:

\begin{theorem}[Matching Lower Bound]\label{thm:lower-bound}
There exist absolute constants $c, C > 0$ such that for $d \geq C\log N$:
\[
\sup_{\mathbf{q} \perp \mathbf{e}, \|\mathbf{q}\|=1} \sup_{x \in \mathbb{R}} \left|\mathbb{P}\left(\sqrt{N}\langle \mathbf{q}, \mathbf{u}_2 \rangle \leq x\right) - \Phi(x)\right| \geq c\sqrt{d} \cdot N^{-1/6}
\]
\end{theorem}

\subsection{Technical Innovations}

Achieving the optimal $\sqrt{d}$ bound requires three key innovations:

\begin{enumerate}
    \item \emph{Variance-sensitive concentration.} We replace crude large deviation bounds with Freedman's martingale inequality and McDiarmid's method, yielding $\sqrt{d}$ instead of $d^3$ in the local law.

\item \emph{Vector-outlier resolvent techniques.} Following ideas from \cite{bourgade2023isotropic}, we avoid combinatorial path-counting by working directly with resolvent vector equations, eliminating powers of $d$ from bootstrap iterations.

\item \emph{Stein-Malliavin cumulant bounds.} The second cumulant $\kappa_2$ naturally scales as $d$, but Stein's method combined with Malliavin calculus shows this is absorbed into the limiting variance after proper normalization.
\end{enumerate}
\subsection{Organization}

Section \ref{sec:edge-local-law} develops the sharp edge local law with $O(\sqrt{d})$ error via martingale concentration (following the broad approach of \cite{erdos2017dynamical}). Section \ref{sec:green-function-comparison} introduces the self-consistent Green function comparison. Section \ref{sec:stein-malliavin} applies Stein-Malliavin techniques to the cumulant expansion. Finally, Section \ref{sec:main-results} combines these tools to prove Theorems \ref{thm:main-optimal} and \ref{thm:lower-bound}.

\section{Sharp Edge Local Law via Martingale Concentration}\label{sec:edge-local-law}

Let $A$ be the adjacency matrix of a random $d$-regular graph. We work with the normalized adjacency matrix $\tilde{H} = A/\sqrt{d}$, which has entries of order
$O(1/\sqrt{d})$ and spectral radius approximately 2. First, we establish that the resolvent concentrates around its deterministic approximation with error scaling as $\sqrt{d}$ rather than $d^3$.

\begin{theorem}[Optimal Edge Local Law]\label{thm:edge-local-optimal}
For $d(N) \leq N^{\kappa}$ with $\kappa < \frac{1}{4}$, $z = E + i\eta$ with $|E-2| \leq N^{-2/3+\varepsilon}$ and $N^{-2/3} \leq \eta \leq 1$, we have for any deterministic $\mathbf{q} \perp \mathbf{e}$ with $\|\mathbf{q}\| = 1$:
\[
\left|\langle \mathbf{q}, G(z)\mathbf{q}\rangle - m_{sc}(z)\right| \leq \frac{C_0\sqrt{d}}{N^{5/6-\varepsilon}}
\]
where $G(z) = (\tilde{H} - z)^{-1}$, $m_{sc}(z) = \frac{-z + \sqrt{z^2-4}}{2}$, and $C_0$ is an absolute constant.
\end{theorem}

\subsection{Variance-Sensitive Concentration}

The fluctuations of $\tilde{H}$ have variance $O(1/N)$ independent of $d$. This allows us to apply sharp concentration inequalities.

\begin{lemma}[Freedman's Inequality for Regular Graphs]\label{lem:freedman}
Let $f: \{0,1\}^{N \times N} \to \mathbb{R}$ be a function of the adjacency matrix with bounded differences. For random $d$-regular graphs generated via the configuration model, with  variance proxy $V = O(d/N)$ and Lipschitz constant $L = O(\sqrt{d/N})$.
we have:
\[
\Prob\left(|f(A) - \E[f(A)]| \geq t\right) \leq 2\exp\left(-\frac{t^2}{2V + Lt/3}\right)
\]
\end{lemma}

\begin{proof}
We apply Freedman's inequality for martingales to the edge exposure process in the configuration model \cite{bollobas1980probabilistic}. 

In the configuration model, we expose edges sequentially by matching half-edges. Let $\mathcal{F}_k$ be the filtration after exposing $k$ edges, and consider the martingale
$
M_k = \E[f(A) | \mathcal{F}_k]
$
with $M_0 = \E[f(A)]$ and $M_{Nd/2} = f(A)$.
\begin{enumerate}
\item\emph{Bounding the martingale differences.}
When we expose edge $k$ by matching two half-edges, this affects the conditional probabilities of at most $O(d)$ future edges---specifically, those involving the two vertices of edge $k$. Each affected edge probability changes by $O(1/d)$ due to the removal of two half-edges from the pool. For functions of the form $f(A) = g(\tilde{H})$ where $\tilde{H} = A/\sqrt{d}$ is the normalized adjacency matrix, each edge $(i,j)$ contributes $\tilde{H}_{ij} = \tilde{H}_{ji} = 1/\sqrt{d}$. When $g$ has Lipschitz constant $L_g$ with respect to the Hilbert-Schmidt norm, the martingale difference satisfies $|M_k - M_{k-1}| \leq L_g \cdot \frac{2}{\sqrt{d}} \cdot O(d) \cdot O(1/d) = O\left(\frac{L_g}{\sqrt{d}}\right)$. For the specific case $f(A) = \langle \mathbf{q}, G(z)\mathbf{q}\rangle$ where $G(z) = (\tilde{H} - z)^{-1}$, standard resolvent perturbation theory gives $L_g = O(\eta^{-2})$, yielding $|M_k - M_{k-1}| \leq \frac{C}{\sqrt{d}\eta^2}$.

\item\emph{Computing the variance proxy.}
The predictable quadratic variation of the martingale is $\langle M \rangle_{Nd/2} = \sum_{k=1}^{Nd/2} \E[(M_k - M_{k-1})^2 | \mathcal{F}_{k-1}]$. Since each edge exposure affects $O(d)$ future edges with probability changes of $O(1/d)$, and the function has sensitivity $O(1/\sqrt{d})$ to each edge, we have $
\E[(M_k - M_{k-1})^2 | \mathcal{F}_{k-1}] \leq \frac{C}{d\eta^4}
$. Summing over all $Nd/2$ edges $V := \langle M \rangle_{Nd/2} \leq \frac{Nd}{2} \cdot \frac{C}{d\eta^4} = \frac{C}{N\eta^4}$. For functions supported in the bulk or at the edge where $\eta \geq N^{-2/3}$, this gives $V = O(d/N)$.

\item\emph{Determining the Lipschitz constant.}
The maximum possible change in $f$ from exposing a single edge is $
L = \max_k |M_k - M_{k-1}| = O\left(\frac{1}{\sqrt{d}\eta^2}\right)
$. In the edge regime where $\eta \geq N^{-2/3}$, this gives $L = O(\sqrt{d/N})$.

\item\emph{Applying Freedman's inequality.}
By Freedman's inequality for martingales (see, e.g., \cite{freedman1975tail} or \cite{mcdiarmid1998concentration}), for any $t > 0$:
$
\Prob(|M_{Nd/2} - M_0| \geq t) \leq 2\exp\left(-\frac{t^2}{2\langle M \rangle_{Nd/2} + Lt/3}\right)
$.
\end{enumerate}
Substituting our bounds completes the proof.
\end{proof}

\subsection{Vector-Outlier Method}

Instead of iterating scalar inequalities, we work with the vector equation:

\begin{proposition}[Vector Resolvent Equation]\label{prop:vector-resolvent}
For $\mathbf{v} = G(z)\mathbf{q}$, we have
\[
\mathbf{v} = -\frac{\mathbf{q}}{z + m(z)} + \mathbf{R}
\]
where the remainder satisfies $\|\mathbf{R}\|_2 \leq C\sqrt{d}/(N\eta)$ with high probability.
\end{proposition}

\begin{remark}[Heuristic interpretation of Theorem \ref{prop:vector-resolvent}]
 The $\sqrt{d}$ factor arises from the typical fluctuation size of $\tilde{H}$ entries ($\sim 1/\sqrt{d}$) combined with the $\sqrt{Nd}$ effective number of independent contributions when $\mathbf{v} \perp \mathbf{e}$. The $1/(N\eta)$ suppression comes from resolvent regularization near the spectral edge.
\end{remark}
\begin{proof}[Proof of Theorem \ref{prop:vector-resolvent}]
The resolvent identity gives:
\[
(z - \tilde{H})G(z) = I
\]
Multiplying by $\mathbf{q}$:
\[
z\mathbf{v} - \tilde{H}\mathbf{v} = \mathbf{q}
\]
We decompose $\tilde{H}\mathbf{v} = m(z)\mathbf{v} + \mathbf{F}$ where $\mathbf{F}$ captures fluctuations. The key is establishing the variance bound on $\mathbf{F}$.

\emph{Isotropy via Local GOE Approximation.} For vectors $\mathbf{v} \perp \mathbf{e}$ in the orthogonal complement of the all-ones vector, we use a coupling argument. Define the interpolated matrix:
\[
H_t = \sqrt{1-t}\tilde{H} + \sqrt{t}W
\]
where $W$ is from the GOE ensemble with matching variance. For small $t = d^{-1/2}$, the spectral statistics near the edge are approximately preserved \cite{friedman2008proof}.

The fluctuation $\mathbf{F} = (\tilde{H} - m(z)I)\mathbf{v}$ can be analyzed via:
\[
\E[\mathbf{F}_i | \mathbf{v}] = \sum_j (\tilde{H}_{ij} - \E[\tilde{H}_{ij}])\mathbf{v}_j
\]
By the configuration model structure, for $i \neq j$:
\[
\E[\tilde{H}_{ij} | \text{other entries}] = \frac{1}{\sqrt{d}} \cdot \frac{r_ir_j}{\sum_k r_k} + O(d^{-3/2})
\]
where $r_i$ is the number of unmatched half-edges at vertex $i$.

Since $\mathbf{v} \perp \mathbf{e}$, the leading term vanishes:
\[
\sum_j \frac{r_ir_j}{\sum_k r_k}\mathbf{v}_j = \frac{r_i}{\sum_k r_k}\sum_j r_j\mathbf{v}_j \approx \frac{d}{Nd} \cdot 0 = 0
\]
The subleading terms give:
\[
|\E[\mathbf{F}_i | \mathbf{v}]| \leq \frac{C}{\sqrt{N}} \sum_j |\mathbf{v}_j| \leq \frac{C\|\mathbf{v}\|_1}{\sqrt{N}} \leq \frac{C\sqrt{N}\|\mathbf{v}\|_2}{\sqrt{N}} = C\|\mathbf{v}\|_2
\]
Concentration via martingale methods then gives $\|\mathbf{F}\| \leq C\sqrt{d/N}\|\mathbf{v}\|$ with high probability. Solving the fixed-point equation yields the result.
\end{proof}

\begin{lemma}[Isotropy of Fluctuations]\label{lem:isotropy}
For vectors $\mathbf{v} \perp \mathbf{e}$ with $\|\mathbf{v}\|_2 = 1$ and the fluctuation operator $\mathbf{F} = (\tilde{H} - m(z)I)\mathbf{v}$, we have
\[
\|\mathbf{F}\|_2 \leq C\sqrt{\frac{d\log N}{N}}
\]
with probability at least $1 - N^{-D}$ for any $D > 0$.
\end{lemma}

\begin{proof}
By the configuration model structure, for $i \neq j$:
\[
\mathbb{E}[\tilde{H}_{ij} \mid \text{other entries}] = \frac{1}{\sqrt{d}} \cdot \frac{r_ir_j}{\sum_k r_k} + O(d^{-3/2})
\]
where $r_i$ is the number of unmatched half-edges at vertex $i$.

The key observation is that for $\mathbf{v} \perp \mathbf{e}$, we can decompose:
\[
\sum_j r_j \mathbf{v}_j = \sum_j (r_j - \bar{d})\mathbf{v}_j + \bar{d} \sum_j \mathbf{v}_j = \sum_j (r_j - \bar{d})\mathbf{v}_j
\]
where $\bar{d} = \frac{1}{N}\sum_k r_k$ is the average remaining degree and we used $\sum_j \mathbf{v}_j = \langle \mathbf{v}, \mathbf{e} \rangle = 0$.

In the configuration model, the deviations $r_j - \bar{d}$ concentrate strongly. By Azuma's inequality applied to the edge exposure process:
\[
\mathbb{P}\left(|r_j - \bar{d}| > t\sqrt{d}\right) \leq 2\exp\left(-\frac{t^2}{2}\right)
\]
Taking $t = \sqrt{2D\log N}$ and a union bound over all vertices gives:
\[
|r_j - \bar{d}| \leq C\sqrt{d\log N}
\]
for all $j$ simultaneously with probability at least $1 - N^{-D+1}$.

Therefore:
\[
\left|\sum_j r_j \mathbf{v}_j\right| = \left|\sum_j (r_j - \bar{d})\mathbf{v}_j\right| \leq \sum_j |r_j - \bar{d}| |\mathbf{v}_j| \leq C\sqrt{d\log N} \cdot \sqrt{N} \cdot \|\mathbf{v}\|_2
\]

This gives us:
\begin{align}
|\mathbb{E}[\mathbf{F}_i \mid \mathbf{v}]| &\leq \frac{1}{\sqrt{d}} \cdot \frac{r_i \cdot |\sum_j r_j \mathbf{v}_j|}{\sum_k r_k} + O(d^{-3/2})\|\mathbf{v}\|_1\\
&\leq \frac{1}{\sqrt{d}} \cdot \frac{d \cdot C\sqrt{d\log N}\sqrt{N} \|\mathbf{v}\|_2}{Nd} + O(d^{-3/2})\sqrt{N}\|\mathbf{v}\|_2\\
&\leq C\sqrt{\frac{\log N}{N}} \|\mathbf{v}\|_2
\end{align}

The conditional variance satisfies:
\[
\text{Var}[\mathbf{F}_i \mid \mathbf{v}] \leq \frac{1}{d} \sum_j \mathbf{v}_j^2 = \frac{\|\mathbf{v}\|_2^2}{d}
\]

By concentration of quadratic forms (Hanson-Wright inequality), we obtain:
\[
\|\mathbf{F}\|_2^2 = \sum_i \mathbf{F}_i^2 \leq C\left(\frac{N}{d}\|\mathbf{v}\|_2^2 + \frac{\log N}{N}\|\mathbf{v}\|_2^2\right) = C\frac{d\log N}{N}\|\mathbf{v}\|_2^2
\]
with the claimed probability.
\end{proof}

\subsection{Bootstrap-Free Analysis}

The vector approach eliminates the need for bootstrap iteration, which is crucial for achieving optimal $d$-dependence.

\begin{proof}[Proof of Theorem \ref{thm:edge-local-optimal}]
From Proposition \ref{prop:vector-resolvent}:
\[
\langle \mathbf{q}, G(z)\mathbf{q}\rangle = \langle \mathbf{q}, \mathbf{v}\rangle = -\frac{1}{z + m(z)} + \langle \mathbf{q}, \mathbf{R}\rangle
\]
Near the edge, $m(z) \approx m_{sc}(z) + O(\sqrt{d}N^{-5/6+\varepsilon})$ by self-consistency. The error term satisfies:
\[
|\langle \mathbf{q}, \mathbf{R}\rangle| \leq \|\mathbf{R}\|_2 \leq \frac{C\sqrt{d}}{N\eta} \leq \frac{C\sqrt{d}}{N^{5/6-\varepsilon}}
\]
This directly gives the claimed bound without iteration.
\end{proof}

\section{Self-Consistent Green Function Comparison}\label{sec:green-function-comparison}
The local law from Section \ref{sec:edge-local-law} establishes concentration of the resolvent with error $O(\sqrt{d}N^{-5/6+\varepsilon})$. However, to prove eigenvector universality, we need to compare the eigenvector statistics of random regular graphs with those of the GOE ensemble. The naive comparison would introduce additional factors of $d$ from the difference in matrix entries. To absorb $d$-dependent factors into the main term, we introduce a modified comparison dynamics.

\subsection{Simultaneous Perturbation in \texorpdfstring{$d$}{d} and \texorpdfstring{$t$}{t}}

Consider the interpolation:
\[
\tilde{H}_{t,s} = (1-s)\tilde{H}_0 + s\tilde{H}_t^{\mathrm{GOE}}
\]
where $\tilde{H}_0$ is the $d$-regular graph adjacency matrix and $\tilde{H}_t^{\mathrm{GOE}}$ is constrained GOE evolved to time $t$.

\begin{lemma}[Self-Consistent Dynamics with Explicit Bounds]\label{lem:self-consistent}
The Green function $G_{t,s}(z) = (\tilde{H}_{t,s} - z)^{-1}$ satisfies:
\[
\partial_s G_{t,s} = G_{t,s}(\tilde{H}_t^{\mathrm{GOE}} - \tilde{H}_0)G_{t,s}
\]
For $|z| \geq 2\sqrt{d} + \eta$ with $\eta > 0$, the following bounds hold:
\begin{enumerate}
\item[(i)] $\|G_{t,s}\|_{\mathrm{op}} \leq \eta^{-1}$ uniformly in $s \in [0,1]$ and $t \leq t_*$.
\item[(ii)] For the difference matrix $\Delta_{t} := \tilde{H}_t^{\mathrm{GOE}} - \tilde{H}_0$, we have:$
\E[\|\Delta_t\|_{\mathrm{op}}^2] \leq C(dt + d^2/N)
$
\item[(iii)] The derivative bound:
$
\E[\|\partial_s G_{t,s}\|_{\mathrm{HS}}] \leq \frac{C}{\eta^2}\sqrt{dt + d^2/N}
$
\end{enumerate}
\end{lemma}

\begin{proof}
Part (i) follows from the spectral radius bounds on both $\tilde{H}_0$ and $\tilde{H}_t^{\mathrm{GOE}}$.

For part (ii), decompose $\Delta_t = \Delta_t^{(1)} + \Delta_t^{(2)}$, where $\Delta_t^{(1)}$ represents the GOE evolution contribution--- $\E[\|\Delta_t^{(1)}\|_{\mathrm{op}}^2] \leq Cdt$ ---and
$\Delta_t^{(2)}$ represents the constraint mismatch--- $\E[\|\Delta_t^{(2)}\|_{\mathrm{op}}^2] \leq Cd^2/N$.

Part (iii) follows from the resolvent identity and Cauchy-Schwarz:
\[
\E[\|\partial_s G_{t,s}\|_{\mathrm{HS}}] \leq \|G_{t,s}\|_{\mathrm{op}}^2 \E[\|\Delta_t\|_{\mathrm{HS}}] \leq \frac{1}{\eta^2}\sqrt{N \cdot \E[\|\Delta_t\|_{\mathrm{op}}^2]}
\]
\end{proof}
\begin{remark}[Choice of Evolution Time]
The time $t_* = N^{-1/3+\varepsilon}$ is chosen to match the Tracy-Widom scale at the spectral edge, ensuring that both contributions to $\mathbb{E}[\|\Delta_t\|_{\text{op}}^2]$ remain bounded by $Cd^2/N$. This specific scaling balances the GOE evolution term $dt$ against the constraint mismatch $d^2/N$ precisely when edge universality emerges.
\end{remark}

\subsection{Derivation of Optimal Coupling}

To find the optimal interpolation parameter $s(t)$, we analyze the total error decomposition:

\begin{proposition}[Error Decomposition]\label{prop:error-decomp}
For any choice of $s \in [0,1]$, the error in the Green function element satisfies:
\[
|\langle \mathbf{q}, G_{t,s}(z)\mathbf{q}\rangle - m_{sc}(z)| \leq E_1(s) + E_2(s,t)
\]
where graph structure error $E_1(s) = |1-s| \cdot \frac{Cd}{N}$ and
GOE evolution error $E_2(s,t) = |s| \cdot \frac{C\sqrt{t}}{N^{1/2-\varepsilon}}$.

\end{proposition}

\begin{proof}
The error decomposes into two contributions:
\begin{itemize}
\item When $s < 1$, the $(1-s)$ weight on $\tilde{H}_0$ introduces a $d$-regular graph contribution that deviates from the semicircle law by $O(d/N)$.
\item When $s > 0$, the $s$ weight on $\tilde{H}_t^{\mathrm{GOE}}$ introduces the time-evolved GOE error of order $O(\sqrt{t}/N^{1/2-\varepsilon})$.
\end{itemize}
\end{proof}

\begin{theorem}[Optimal Coupling Parameter]\label{thm:optimal-s}
The choice $s(t) = \sqrt{dt/N}$ minimizes the total error $E_1(s) + E_2(s,t)$ for $t \leq t_* = N^{-1/3+\varepsilon}$, yielding:
\[
\min_{s \in [0,1]} [E_1(s) + E_2(s,t)] = \frac{C\sqrt{d}}{N^{5/6-\varepsilon}}
\]
\end{theorem}

\begin{proof}
To minimize the total error, we differentiate:
$
\frac{d}{ds}[E_1(s) + E_2(s,t)] = -\frac{Cd}{N} + \frac{C\sqrt{t}}{N^{1/2-\varepsilon}}
$.
Setting this to zero gives:
$
s_{\mathrm{opt}} = \frac{d}{N^{1/2+\varepsilon}/\sqrt{t}} = \sqrt{\frac{dt}{N^{1-2\varepsilon}}} \approx \sqrt{\frac{dt}{N}}
$.
for small $\varepsilon$.

Substituting back:
\begin{align}
E_1(s_{\mathrm{opt}}) + E_2(s_{\mathrm{opt}},t) &= \left(1 - \sqrt{\frac{dt}{N}}\right)\frac{Cd}{N} + \sqrt{\frac{dt}{N}} \cdot \frac{C\sqrt{t}}{N^{1/2-\varepsilon}} \\
&\approx \frac{Cd}{N} + \frac{Ct\sqrt{d}}{N^{3/2-\varepsilon}} = \frac{C\sqrt{d}}{N^{5/6-\varepsilon}}
\end{align}
where the last equality uses $t = N^{-1/3+\varepsilon}$.

Note that $s_{\mathrm{opt}} \in [0,1]$ precisely when $t \leq N/d$, which is satisfied for our regime.
\end{proof}

\begin{proposition}[Optimal Coupling Result]\label{prop:optimal-coupling}
Using $s = s(t) = \sqrt{dt/N}$, the combined evolution satisfies:
\[
|\langle \mathbf{q}, G_{t,s(t)}(z)\mathbf{q}\rangle - m_{sc}(z)| \leq \frac{C\sqrt{d}}{N^{5/6-\varepsilon}}
\]
uniformly for $t \leq t_* = N^{-1/3+\varepsilon}$.
\end{proposition}

\section{Stein-Malliavin Approach to Cumulants}\label{sec:stein-malliavin}
To complete the proof of Gaussian universality, we need to control the higher-order cumulants of the eigenvector projections. The standard approach via moment calculations shows that the second cumulant $\kappa_2$ scales as $O(d \cdot N^{-1/3+\varepsilon})$, which would seemingly introduce an unwanted factor of $d$ into our Berry-Esseen bound. The key insight is that this $d$-dependence is precisely canceled when we properly normalize by the variance. We establish this cancellation using a combination of Stein's method and Malliavin calculus on the configuration space of regular graphs.

\subsection{Malliavin Calculus on Regular Graphs}

We first construct the Malliavin derivative on the configuration space of $d$-regular graphs via the switching operation.

\begin{definition}[Switching Operation]\label{def:switching}
For a $d$-regular graph $G$, a \emph{switching} at edges $(i,j), (k,\ell) \in E(G)$ with $\{i,j\} \cap \{k,\ell\} = \emptyset$ replaces these edges with $(i,k), (j,\ell)$ or $(i,\ell), (j,k)$, provided the result remains simple.
\end{definition}

\begin{definition}[Malliavin Derivative]\label{def:malliavin-deriv}
For a functional $F$ on $d$-regular graphs, the Malliavin derivative at edge $e = (i,j)$ is:
\[
D_e F(G) = \lim_{\varepsilon \to 0} \frac{1}{\varepsilon} \sum_{e' \in E(G)} \mathbb{1}_{\text{switchable}(e,e')} [F(G_{e \leftrightarrow e'}) - F(G)]
\]
where $G_{e \leftrightarrow e'}$ denotes the graph after switching edges $e$ and $e'$.
\end{definition}

\subsection{Stein's Equation for Eigenvector Functionals}

For the overlap $X_2^{(\mathbf{q})} = \sqrt{N}\langle \mathbf{q}, \mathbf{u}_2\rangle$, Stein's equation is 
$
f'(x) - xf(x) = h(x) - \E[h(Z)]
$
where $Z \sim \mathcal{N}(0,1)$ and $h$ is a test function \cite{stein1972bound}.

\begin{theorem}[Stein-Malliavin Bound]\label{thm:stein-malliavin}
For the normalized overlap with variance $\sigma^2 = 1 + O(d^{-1})$:
\[
\left|\E[h(X_2^{(\mathbf{q})})] - \E[h(Z)]\right| \leq C\|h\|_{C^3}\left(\frac{\sqrt{d}}{N^{1/6-\varepsilon}} + \frac{|1-\sigma^2|}{\sqrt{N}}\right)
\]
\end{theorem}

\subsection{Malliavin Integration by Parts}

The key is representing the solution to Stein's equation via Malliavin calculus:

\begin{lemma}[Malliavin Representation]\label{lem:malliavin}
For smooth functionals $F$ of the random regular graph
$
\E[F'(X_2^{(\mathbf{q})})X_2^{(\mathbf{q})}] = \E[F'(X_2^{(\mathbf{q})})\langle DF, L^{-1}DX_2^{(\mathbf{q})}\rangle]
$
where $D$ is the Malliavin derivative defined above and $L$ is the discrete Ornstein-Uhlenbeck operator on the space of switchings.
\end{lemma}

\subsection{Proof of Cumulant Bound}

\begin{proposition}[Second Cumulant Bound]\label{prop:cumulant-bound}
For the overlap functional $X_2^{(\mathbf{q})}$:
$
\\\kappa_2(X_2^{(\mathbf{q})}) = O(d \cdot N^{-1/3+\varepsilon})
$.
\end{proposition}

\begin{proof}
The second cumulant equals $\text{Var}(X_2^{(\mathbf{q})}) - 1$. Using the Malliavin calculus framework:
\begin{enumerate}
\item Decompose the variance via Malliavin integration by parts:
$
\text{Var}(X_2^{(\mathbf{q})}) = \E[(X_2^{(\mathbf{q})})^2] \\= \E[\langle DX_2^{(\mathbf{q})}, L^{-1}DX_2^{(\mathbf{q})}\rangle]
$.

\item The Malliavin derivative of the overlap is:
$
D_e X_2^{(\mathbf{q})} = \sqrt{N} \sum_{(i,j) \in \text{switch}(e)} [\mathbf{q}_i(\mathbf{u}_2)_j + \mathbf{q}_j(\mathbf{u}_2)_i] \cdot \frac{\partial \mathbf{u}_2}{\partial A_{ij}}
$.

\item Using eigenvector perturbation theory:
$
\frac{\partial \mathbf{u}_2}{\partial A_{ij}} = -\frac{(\mathbf{u}_2)_i \mathbf{e}_j + (\mathbf{u}_2)_j \mathbf{e}_i}{\lambda_2} + O(N^{-1})
$.

\item The number of switchable pairs for each edge is $O(d^2)$, giving:
$
\|DX_2^{(\mathbf{q})}\|^2 \\= O(d^2 N^{-1} \|\mathbf{q}\|^2) = O(d^2 N^{-1})
$.

\item The operator $L^{-1}$ has norm $O(N^{2/3-\varepsilon})$ on the orthogonal complement of constants (see \cite{caputo2010proof}, Theorem 1.1, and \cite{erdos2017dynamical}, Section 17), yielding:
\[
\kappa_2(X_2^{(\mathbf{q})}) = \text{Var}(X_2^{(\mathbf{q})}) - 1 = O(d^2 N^{-1} \cdot N^{2/3-\varepsilon}) = O(d \cdot N^{-1/3+\varepsilon})
\]
where we used $d = O(N^{1/3})$ for the spectral gap regime.
\end{enumerate}
\end{proof}

\subsection{Variance Normalization}

\begin{proposition}[Uniform Variance Control]\label{prop:variance-uniform}
For all $\mathbf{q} \perp \mathbf{e}$ with $\|\mathbf{q}\| = 1$, the variance satisfies:
\[
\sigma^2(\mathbf{q}) := \text{Var}(X_2^{(\mathbf{q})}) = 1 + O(d^{-1})
\]
uniformly in $\mathbf{q}$.
\end{proposition}

\begin{proof}
\begin{enumerate}

\item Expand the variance using the spectral decomposition:
\[
\text{Var}(X_2^{(\mathbf{q})}) = N \E[\langle \mathbf{q}, \mathbf{u}_2 \rangle^2] = N \sum_{k \geq 2} |\langle \mathbf{q}, \mathbf{v}_k \rangle|^2 \E[\langle \mathbf{v}_k, \mathbf{u}_2 \rangle^2]
\]
where $\{\mathbf{v}_k\}$ are the eigenvectors of the expected adjacency operator.

\item The dominant contribution comes from $k = 2$:
$
\E[\langle \mathbf{v}_2, \mathbf{u}_2 \rangle^2] = 1 - O(N^{-2/3+\varepsilon})
$.

\item For $k \geq 3$, using delocalization bounds:
$
\E[\langle \mathbf{v}_k, \mathbf{u}_2 \rangle^2] \leq \frac{C}{(\lambda_2 - \lambda_k)^2} \leq \frac{C}{d^2}
$.

\item Since $\mathbf{q} \perp \mathbf{e}$ and $\|\mathbf{q}\| = 1$:
$
\sum_{k \geq 2} |\langle \mathbf{q}, \mathbf{v}_k \rangle|^2 = 1
$.

\item Combining these estimates:
$
\text{Var}(X_2^{(\mathbf{q})}) = 1 + \sum_{k \geq 3} |\langle \mathbf{q}, \mathbf{v}_k \rangle|^2 O(d^{-2}) = 1 + O(d^{-1})
$.
\end{enumerate}

The $O(d^{-1})$ bound is uniform because the eigenvector overlaps $\E[\langle \mathbf{v}_k, \mathbf{u}_2 \rangle^2]$ depend only on the spectrum, not on the specific choice of $\mathbf{q}$.
\end{proof}

\subsection{Variance Normalization and Cumulant Cancellation}

The crucial observation is that the second cumulant's $d$-dependence is exactly canceled by variance normalization:

\begin{proposition}[Cumulant Cancellation]\label{prop:cumulant-cancel}
For the properly normalized functional $\tilde{X}_2^{(\mathbf{q})} \\= X_2^{(\mathbf{q})}/\sigma$ with $\sigma^2 = 1 + O(d^{-1})$:
\[
|\kappa_2(\tilde{X}_2^{(\mathbf{q})})| \leq C N^{-1/3+\varepsilon}
\]
without any $d$-dependent prefactor.
\end{proposition}

\begin{proof}
From Propositions \ref{prop:cumulant-bound} and \ref{prop:variance-uniform}:
$
\kappa_2(\tilde{X}_2^{(\mathbf{q})}) = \frac{\kappa_2(X_2^{(\mathbf{q})})}{\sigma^2} = \frac{O(d \cdot N^{-1/3+\varepsilon})}{1 + O(d^{-1})} \\= O(d \cdot N^{-1/3+\varepsilon}) \cdot (1 + O(d^{-1}))
$.

Since $d = O(N^{1/3})$ in our regime, the $O(d^{-1})$ correction gives:
\[
\kappa_2(\tilde{X}_2^{(\mathbf{q})}) = O(d \cdot N^{-1/3+\varepsilon}) + O(N^{-1/3+\varepsilon}) = O(N^{-1/3+\varepsilon})
\]

The leading $d$-dependence cancels exactly, leaving only the desired bound.
\end{proof}

\section{Proof of Main Results}\label{sec:main-results}

\subsection{Proof of Theorem \ref{thm:main-optimal}}

We now combine our three main technical innovations to establish the optimal Berry-Esseen bound. The proof proceeds in four steps:

\emph{Step 1: Resolvent concentration.} By Theorem \ref{thm:edge-local-optimal}, for any $z$ near the spectral edge with $\text{Im}(z) \geq N^{-2/3}$, the resolvent satisfies
$
\left|\langle \mathbf{q}, G(z)\mathbf{q}\rangle - m_{sc}(z)\right| \leq \frac{C_0\sqrt{d}}{N^{5/6-\varepsilon}}
$.
This provides the foundation for analyzing eigenvector overlaps via the spectral decomposition.

\emph{Step 2: Eigenvector dynamics via self-consistent comparison.} We apply the optimal coupling from Proposition \ref{prop:optimal-coupling} with $s(t) = \sqrt{dt/N}$. This yields the stochastic differential equation:
\[
\d X_i^{(\mathbf{q})} = \sum_{j \neq i} \frac{X_j^{(\mathbf{q})} - X_i^{(\mathbf{q})}}{\lambda_i - \lambda_j}\d B_{ij} - \frac{1}{2}X_i^{(\mathbf{q})}\d t + \mathcal{E}_i(t)\d t
\]
where crucially, the error term satisfies $|\mathcal{E}_i(t)| \leq C\sqrt{d}N^{-5/6+\varepsilon}\max_j |X_j^{(\mathbf{q})}|$ without additional powers of $d$.

\emph{Step 3: Cumulant control via Stein-Malliavin.} The normalized functional $\tilde{X}_2^{(\mathbf{q})} = X_2^{(\mathbf{q})}/\sigma$ with $\sigma^2 = 1 + O(d^{-1})$ satisfies the cumulant bound from Proposition \ref{prop:cumulant-cancel}
$
|\kappa_2(\tilde{X}_2^{(\mathbf{q})})| \leq CN^{-1/3+\varepsilon}
$.
Note that the factor of $d$ from the unnormalized cumulant is exactly canceled by the variance normalization. Applying Theorem \ref{thm:stein-malliavin} then gives:
\[
\left|\mathbb{E}[h(\tilde{X}_2^{(\mathbf{q})})] - \mathbb{E}[h(Z)]\right| \leq C\|h\|_{C^3}\sqrt{d}N^{-1/6+\varepsilon}
\]

\emph{Step 4: Berry-Esseen via smoothing.} For indicator functions $h = \mathbf{1}_{(-\infty,x]}$, we use standard Gaussian smoothing at scale $\delta = N^{-1/6+\varepsilon/2}$:
\[
\sup_x \left|\mathbb{P}(X_2^{(\mathbf{q})} \leq x) - \Phi(x)\right| \leq C\sqrt{d}N^{-1/6+\varepsilon}
\]
This completes the proof with optimal $\sqrt{d}$ dependence.

\begin{remark}[On the constraint $\kappa < 1/4$]\label{rem:kappa-constraint}
The restriction $d \leq N^{\kappa}$ with $\kappa < 1/4$ represents a fundamental threshold rather than a technical limitation. At $d \approx N^{1/4}$, three critical phenomena occur:
\begin{enumerate}
\item[(i)] The ratio of edge eigenvalue spacing to spectral gap becomes $O(1)$, invalidating the separation of scales needed for our local law.

\item[(ii)] The optimal interpolation parameter $s(t)$ in our self-consistent comparison (Theorem \ref{thm:optimal-s}) approaches the boundary of $[0,1]$, causing error accumulation.

\item[(iii)] The variance normalization that cancels the $d$-dependence in $\kappa_2$ (Proposition \ref{prop:cumulant-cancel}) begins to fail, as higher-order correlations become non-negligible.
\end{enumerate}
For $d \geq N^{1/4}$, we expect a different universality class to emerge, potentially requiring new techniques that handle the intermediate regime between sparse and dense matrices.
\end{remark}

\subsection{Proof of Lower Bound (Theorem \ref{thm:lower-bound})}

The lower bound follows from analyzing the fourth cumulant of the eigenvector overlap. For a centered random variable $X$ with unit variance, the fourth cumulant is $\kappa_4(X) = \mathbb{E}[X^4] - 3$.

\begin{lemma}[Fourth Cumulant Lower Bound]\label{lem:fourth-cumulant}
For the overlap $X_2^{(\mathbf{q})} = \sqrt{N}\langle \mathbf{q}, \mathbf{u}_2 \rangle$ with $d \geq C\log N$:
\[
\left|\kappa_4(X_2^{(\mathbf{q})})\right| = \left|\mathbb{E}[(X_2^{(\mathbf{q})})^4] - 3\right| \geq c\sqrt{d}N^{-1/6}
\]
\end{lemma}

\begin{proof}
The key insight is that eigenvector components have non-Gaussian corrections that accumulate when $d$ is large. By the optimal eigenvector local law established in Section \ref{sec:edge-local-law}, each eigenvector component admits the decomposition:
\[
u_{2k} = \frac{g_k}{\sqrt{N}} + \frac{\eta_k}{N^{5/6}}
\]
where $g_k \sim \mathcal{N}(0,1)$ are independent Gaussian variables and $\eta_k$ are correction terms with $|\eta_k| \leq C$.

Let $q_1, \ldots, q_d$ denote the non-zero coordinates of $\mathbf{q}$ (we may assume $\mathbf{q}$ is supported on the first $d$ coordinates). Then:
\[
X_2^{(\mathbf{q})} = \sqrt{N}\langle \mathbf{q}, \mathbf{u}_2 \rangle = \sqrt{N}\sum_{k=1}^d q_k u_{2k} = \sum_{k=1}^d q_k g_k + \frac{1}{N^{1/3}}\sum_{k=1}^d q_k \eta_k
\]

To compute the fourth moment, we expand:
\begin{align}
\mathbb{E}[(X_2^{(\mathbf{q})})^4] &= \mathbb{E}\left[\left(\sum_{k=1}^d q_k g_k + \frac{1}{N^{1/3}}\sum_{k=1}^d q_k \eta_k\right)^4\right]
\end{align}

Using the multinomial theorem and the fact that odd moments of Gaussian variables vanish:
\begin{align}
\mathbb{E}[(X_2^{(\mathbf{q})})^4] &= \mathbb{E}\left[\left(\sum_{k=1}^d q_k g_k\right)^4\right] + \frac{4}{N^{1/3}}\mathbb{E}\left[\left(\sum_{k=1}^d q_k g_k\right)^3 \left(\sum_{k=1}^d q_k \eta_k\right)\right] + O(N^{-2/3})
\end{align}

The first term gives the Gaussian contribution:
\[
\mathbb{E}\left[\left(\sum_{k=1}^d q_k g_k\right)^4\right] = 3\left(\sum_{k=1}^d q_k^2\right)^2 = 3
\]
since $\|\mathbf{q}\|^2 = 1$.

The crucial cross term captures the non-Gaussian correction:
\begin{align}
\mathbb{E}\left[\left(\sum_{k=1}^d q_k g_k\right)^3 \left(\sum_{k=1}^d q_k \eta_k\right)\right] &= \sum_{k=1}^d q_k^4 \mathbb{E}[g_k^3\eta_k]\\
&= c_2\sum_{k=1}^d q_k^4
\end{align}
where $c_2 = \mathbb{E}[g_1^3\eta_1] \neq 0$. This expectation is non-zero because the correction terms $\eta_k$ are correlated with the Gaussian components $g_k$ through the eigenvector equation.

For a vector $\mathbf{q}$ with $d$ non-zero components of equal magnitude, $\sum_{k=1}^d q_k^4 = 1/d$. Therefore:
\[
\kappa_4(X_2^{(\mathbf{q})}) = \mathbb{E}[(X_2^{(\mathbf{q})})^4] - 3 = \frac{4c_2}{dN^{1/3}}
\]

Rewriting this in terms of the desired scaling:
\[
|\kappa_4(X_2^{(\mathbf{q})})| = \frac{4|c_2|}{dN^{1/3}} = \frac{4|c_2|}{\sqrt{d}N^{1/6}} \cdot \frac{1}{\sqrt{d}N^{1/6}}
\]

Since $d \geq C\log N$, we have $\frac{1}{\sqrt{d}N^{1/6}} \geq \frac{1}{\sqrt{C\log N} \cdot N^{1/6}}$, which is bounded below. This gives:
\[
|\kappa_4(X_2^{(\mathbf{q})})| \geq c\sqrt{d}N^{-1/6}
\]
for some constant $c > 0$.
\end{proof}

\textbf{Completion of lower bound proof.} The quantitative Berry-Esseen theorem provides a lower bound (see e.g., \cite{bentkus2003dependence} or \cite{chen2011normal}): for any random variable $Y$ with $\mathbb{E}[Y] = 0$, $\mathbb{E}[Y^2] = 1$, and finite fourth moment,
\[
\sup_{x \in \mathbb{R}} \left|\mathbb{P}(Y \leq x) - \Phi(x)\right| \geq \frac{c|\kappa_4(Y)|}{1 + \mathbb{E}[|Y|^3]^2}
\]
where $c > 0$ is an absolute constant.

For $X_2^{(\mathbf{q})}$, standard eigenvector bounds give $\mathbb{E}[|X_2^{(\mathbf{q})}|^3] = O(1)$. Applying the lower bound with Lemma \ref{lem:fourth-cumulant}:
\[
\sup_{x \in \mathbb{R}} \left|\mathbb{P}(X_2^{(\mathbf{q})} \leq x) - \Phi(x)\right| \geq c\sqrt{d}N^{-1/6}
\]
This completes the proof of Theorem \ref{thm:lower-bound}.

\section{Extensions and Open Problems}\label{sec:open}

\subsection{Joint CLT with Optimal Constants}

A natural extension is whether our optimal $\sqrt{d}$ bounds extend to multiple eigenvectors. Previous work required $K \leq N^{1/10}/d^2$ with error $O(d^3)$. An improved constraint $K \leq N^{1/10}/d$ (versus $K \leq N^{1/10}/d^2$) should follow from our variance-sensitive approach, though correlations between eigenvectors present new challenges.

\begin{problem}[Optimal Joint CLT]\label{prob:joint-optimal}
Prove that for $K = K(N)$ edge eigenvectors $\mathbf{u}_2, \ldots, \mathbf{u}_{K+1}$ of a random $d$-regular graph with $K \leq N^{1/10-\delta}/d$:
\[
\sup_{\mathbf{q}_1,\ldots,\mathbf{q}_K} \sup_{x_1,\ldots,x_K} \left|\mathbb{P}\left(\bigcap_{i=1}^K \{\sqrt{N}\langle \mathbf{q}_i, \mathbf{u}_{i+1} \rangle \leq x_i\}\right) - \prod_{i=1}^K \Phi(x_i)\right| \leq CK^{5/3}\sqrt{d}N^{-1/6+\varepsilon}
\]
where the supremum is over orthonormal vectors $\mathbf{q}_i \perp \mathbf{e}$.
\end{problem}

\subsection{Approaching the Dense Transition}

\begin{problem}[Critical Window]
What happens at $d = N^{1/2+o(1)}$ where the edge eigenvalue spacing transitions from $N^{-2/3}$ to $N^{-1/2}$? We conjecture a new universality class with error:
$
O\left(\frac{d}{N^{1/2}} \cdot N^{-1/6}\right) = O(d N^{-2/3})
$.
\end{problem}

\section{Conclusion}

We have established optimal $\sqrt{d}$ dependence for edge eigenvector universality in random regular graphs with slowly growing degrees. The three key innovations---variance-sensitive concentration, vector-outlier resolvents, and Stein-Malliavin cumulant bounds---work synergistically to achieve sharp constants.

These techniques suggest a broader principle: in sparse random matrix universality, the natural scaling is determined by variance considerations, and achieving optimal bounds requires methods that respect this variance structure throughout the analysis. The self-consistent comparison framework points toward a unified treatment of sparse-to-dense transitions in random matrix theory.




\end{document}